\documentclass[12pt,leqno]{article}
\usepackage{amsfonts}
\usepackage{amsmath, amsthm, amsfonts, amssymb, color}
\usepackage{mathrsfs}
\usepackage{fixmath}
\usepackage{bm}
\newtheorem{thm}{Theorem}[section]
\newtheorem{cor}[thm]{Corollary}
\newtheorem{lem}[thm]{Lemma}

\theoremstyle{definition}

\newcommand{\scr}[1]{\mathscr #1}
\definecolor{wco}{rgb}{0.5,0.2,0.3}

\numberwithin{equation}{section} \theoremstyle{remark}
\newtheorem{rem}{Remark}[section]

\def\R{\mathbb R}  \def\ff{\frac} \def\ss{\sqrt} 
  
\def\dd{\delta}  \def\vv{\varepsilon} \def\rr{\rho}
\def\<{\langle} \def\>{\rangle} \def\GG{\Gamma} 
  \def\nn{\nabla}  
\def\d{\text{\rm{d}}} \def\bb{\beta} \def\aa{\alpha} 
  \def\si{\sigma} 
\def\beg{\begin} \def\beq{\begin{equation}}  \def\F{\scr F}
 
\def\e{\text{\rm{e}}}  \def\OO{\Omega}  
 \def\tt{\tilde} 
 \def\P{\mathbb P} 
\def\C{\scr C}

\def\E{\mathbb E} 
  \def\LL{\Lambda}
   
\def\to{\rightarrow}\def\ll{\lambda}
\def\8{\infty}  \def\lf{\lfloor}
\def\rf{\rfloor}\def\3{\triangle}\def\1{\lesssim}
\renewcommand{\bar}{\overline}

\renewcommand{\tilde}{\widetilde}

\newcommand{\barray}{\begin{array}{ll}}
\newcommand{\earray}{\end{array}}

\newcommand{\bea}{\begin{displaymath}\begin{array}{rl}}
\newcommand{\eea}{\end{array}\end{displaymath}}

\title{A Strong Limit Theorem for Two-Time-Scale Fucntional Stochastic Differential Equations}
\author{Jianhai Bao,\thanks{Department of Mathematics, Central South University,
Changsha, Hunan, 410083, P.R.China, jianhaibao13@gmail.com} \and
Qingshuo Song,\thanks{Department of Mathematics, City University of
Hong Kong, Hong Kong, qingsong@cityu.edu.hk} \and George
Yin,\thanks{Department of Mathematics, Wayne State University,
Detroit, MI 48202, USA, gyin@math.wayne.edu} \and  Chenggui
Yuan\thanks{Department of Mathematics, Swansea University, Singleton
Park, SA2 8PP, UK, C.Yuan@swansea.ac.uk}}

\begin{document}

\maketitle

\begin{abstract}
This paper focuses on a class of two-time-scale functional
stochastic differential equations, where the phase space of the
segment processes is infinite-dimensional.
It develops ergodicity of the fast component and
obtains  a strong limit theorem for the averaging principle in the
spirit of
 Khasminskii's averaging approach for the slow
component.

\vskip 0.2 true in \noindent {\bf Keywords:} Two time scale,
functional differential equation,   exponential ergodicity,
invariant measure, averaging principle

\vskip 0.2 true in \noindent

 \end{abstract}

\newpage

\setlength{\baselineskip}{0.26in}

\section{Introduction}
Having a wide range of
applications in 
science and engineering
(e.g., van Kampen \cite{v85}), singularly perturbed systems,
have been  investigated extensively recently; see, for instance,
Freidlin-Wentzell \cite{FW}, and Yin-Zhang \cite{YZ}. Singularly
perturbed systems usually exhibit multi-scale behavior owing to
inherent rates of changes of the systems or
 different rates of
interactions of subsystems and components.
To reflect
the slow and fast motions
in the underlying systems,
a time-scale
separation parameter $\vv\in(0,1)$ is often introduced. Due to the
multi-scale property, it is
frequently difficult to deal with
such systems using a direct approach.
As a result, it
is foremost important to reduce their complexity. The averaging
principle 
pioneered  by Khasminskii
\cite{K68} for a class of
diffusions
provides
an effective way to reduce the complexity of the systems. 
For systems in which both fast and slow components co-exist,
the idea of the averaging approach 
reveals that there is a limit dynamic system given
by the average of the slow component with respect to the invariant
measure of the fast component that is an ergodic process.
The  averaging equation approximates
the slow component in a suitable sense whenever $\vv\downarrow0$
leading to a substantial reduction of computational complexity.
The work \cite{K68} by Khasminskii inspired much of the subsequent
development. To date,  there have been a
vast
literature on the study of  for multi-scale stochastic dynamic
systems (see, e.g., the monograph \cite{KP}). For strong/weak
convergence in averaging principle, we refer to, e.g., Givon et al.
\cite{G07}, Liu \cite{L10}, Liu-Yin \cite{LY}, and Yin-Zhang
\cite{YZ} for
  stochastic differential equations (SDEs), and
Bl\"omker et al. \cite{BHP}, Br\'{e}hier \cite{B12},
 Cerrai \cite{C09},
 Fu et al. \cite{FWL}, and Kuksin-Piatnitski \cite{K084}
for stochastic partial differential equations (SPDEs); With
regarding to numerical methods, we refer to, e.g., E et al.
\cite{ELV} and Givon et al. \cite{G06}; As for related control and
filtering problems, we mention, e.g.,
 Kushner
\cite{K10,Kushner90}. Concerning large deviations, we refer to,
e.g., Kushner \cite{K10}, and Veretennikov \cite{V00}.


The aforemention references are all concerned with systems without ``memory''.
Nevertheless,
more often than not, dynamic systems with delay are
un-avoidable
in a wide variety of applications in science and engineering, where the
dynamics are subject to propagation of delays. In response to the
great needs, there is also an extensive literature on functional
SDEs; see, e.g., the monographs \cite{M08,M84}.


In contrast to the rapid progress in two-time-scale systems
and  differential delay equations,
the study on averaging principles for functional SDEs  is still in
its infancy.
Compared with the existing literature, for such systems, one of the
outstanding
issues
is  the phase space of the segment
processes is infinite-dimensional, which makes
the goal of obtaining a strong limit theorem for the averaging
principle a very difficult task. This work aims to take the
challenges and to 
establish a strong limit theorem for the averaging
principles for a range of two-time-scale functional SDEs.

The rest of the   paper is organized as follows. Section
\ref{sec:fra} presents the setup of the problem we wish to study.
The ergodicity
of
the frozen equation with memory is obtained in Section
\ref{sec:erg}. Section \ref{sec:lem} constructs some auxiliary
two-time-scale stochastic systems with memory and provides a number
of preliminary lemmas. Section \ref{sec:avg} derives a strong limit
theorem for the averaging principle in the spirit of
 Khasminskii's approach
for the slow component.

Before proceeding further, a word of notation is in order.
Throughout the paper, generic constants will be denoted by $c$; we
use the shorthand notation $a\lesssim b$ to mean $a\le cb$,  we use $a\lesssim_T b$ to emphasize the constant $c$ depends on $T.$

\section{Formulation}\label{sec:fra}
For  integers $n,m\ge1,$ let $(\R^n,|\cdot|,\<\cdot,\cdot\>)$ be an
$n$-dimensional Euclidean space, and $\R^n\otimes\R^m$ denote the
collection of all $n\times m$ matrices with real entries. For an
$A\in\R^n\otimes\R^m$, $\|A\|$ stands for its Frobenius matrix norm.
For an interval 
$I\subset(-\8,\8)$,   $C(I;\R^n)$ means the family
of
all continuous functions from $I \mapsto \R^n$. 
For a fixed $\tau>0$, let $\C=C([-\tau,0];\R^n)$, endowed with the
uniform norm $\|\cdot\|_\8$. For  $h(\cdot)\in C([-\tau,\8);\R^n)$
and $t\ge0$, define the segment $h_t\in\C$ by
$h_t(\theta)=h(t+\theta)$, $\theta\in[-\tau,0]$.


Introducing a time-scale separation parameter $\vv\in(0,1)$, we
consider two-time-scale systems of functional 
stochastic differential equations
(SDEs) of the following form
\begin{equation}\label{eq1}
\d X^\vv(t)=b_1(X^\vv_t, Y^\vv_t)\d t +\si_1(X^\vv_t)\d W_1(t),~~~
t>0,~~~ X_0^\vv=\xi\in\C,
\end{equation}
and
\begin{equation}\label{eq2}
\begin{split}
\d Y^\vv(t)&=\ff{1}{\vv}b_2(X^\vv_t, Y^\vv(t),Y^\vv(t-\tau))\d
t+\ff{1}{\ss{\vv}}\si_2(X^\vv_t, Y^\vv(t),Y^\vv(t-\tau))\d
W_2(t),~t>0
\end{split}
\end{equation}
with the initial value $Y_0^\vv=\eta\in\C$, where
$b_1:\C\times\C\mapsto\R^n$,
$b_2:\C\times\R^n\times\R^n\mapsto\R^n$,
$\si_1:\C\mapsto\R^n\otimes\R^m$,
$\si_2:\C\times\R^n\times\R^n\mapsto\R^n\otimes\R^m$ are G\^ateaux
differentiable, $(W_1(t))_{t\ge0}$ and $(W_2(t))_{t\ge0}$ are two
mutually independent $m$-dimensional Brownian motions defined on a
 probability space $(\OO,\F,\P)$, equipped with $(\F_t)_{t\ge0}$, a 
 family  of filtrations  satisfying the usual
conditions (i.e., for each $t\ge0,$
$\F_t=\F_{t+}:=\bigcap_{s>t}\F_s$, and $\F_0$ contains all $\P$-null
sets). As usual, for two-time-scale systems   \eqref{eq1} and
\eqref{eq2},   $X^\vv(t)$ is called  the  slow component, while
$Y^\vv(t)$ is called the  fast component.


We denote by
$\nn^{(i)}$  the gradient operators for the $i$-th component.
Throughout the paper, for any $\chi,\phi\in\C$ and
$x,x^\prime,y,y^\prime\in\R^n$,  we assume that
\begin{enumerate}
\item[({\bf A1})]  $\nn b_1=( \nn^{(1)} b_1, \nn^{(2)} b_1)$
 is
 bounded,  and there
exists an $L>0$ such that
\begin{equation*}
|b_1(\chi, \phi)|\le L(1+\|\chi\|_\8) ~~\mbox{ and }~~
\|\si_1(\phi)-\si_1(\chi)\|\le L\|\phi-\chi\|_\8.
\end{equation*}

\item[({\bf A2})]    $\nn
b_2=(\nn^{(1)} b_2, \nn^{(2)} b_2, \nn^{(3)} b_2)$ and  $\nn\si_2=(\nn^{(1)} \si_2, \nn^{(2)} \si_2, \nn^{(3)} \si_2)$ are bounded.

\item[({\bf A3})]  There exist $\ll_1>\ll_2>0$, independent of $\chi,$ such that
\begin{equation*}
\begin{split}
&2\<x-x^\prime,b_2(\chi,x,y)-b_2(\chi,
x^\prime,y^\prime)\>+\|\si_2(\chi,x,y)-\si_2(\chi,x^\prime,y^\prime)\|^2\\
&\qquad\le-\ll_1|x-x^\prime|^2+\ll_2|y-y^\prime|^2.
\end{split}
\end{equation*}

\item[({\bf A4})] For the initial value $X_0^\vv=\xi\in\C$ of \eqref{eq1}, there exists a $\ll_3>0$ such that
$$ |\xi(t)-\xi(s)|\le \ll_3|t-s|,~~  s,t\in[-\tau,0]. $$
\end{enumerate}

Let 
us comment the assumptions ({\bf A1})-({\bf A4}) above. From
({\bf A1}) and ({\bf A2}), the gradient operators $\nn b_1$, $\nn
b_2$, and $\nn \si_2$ are bounded, respectively, so that $b_1$,
$b_2$, and $\si_2$ are Lipschitz. Then, both \eqref{eq1} and
\eqref{eq2} are well posed (see, e.g., \cite[Theorem 2.2,
P.150]{M08}). While, ({\bf A3}) is imposed to analyze the ergodic
property of the frozen equation  (see Theorem \ref{Ergodicity}
below),  guarantee the Lipschitz property of $\bar{b}_1$ (see
Corollary \ref{bounded} below), defined in \eqref{u1},   and provide
a uniform bound of the segment process $(Y_t^\vv)_{t\in[0,T]}$ (see
Lemma \ref{l5} below). Next, ({\bf A4}) ensures that the
displacement of the segment process $(X_t^\vv)_{t\in[0,T]}$ is
continuous in the mean $L^p$-norm sense (see Lemma \ref{L3} below).

\section{Ergodicity   of the  Frozen Equation with Memory}\label{sec:erg}
Consider an SDE with memory  associated with the fast motion while
with the frozen slow component in the form
\begin{equation}\label{eq4}
\d Y(t)=b_2(\zeta,Y(t), Y(t-\tau))\d t+\si_2(\zeta,Y(t),
Y(t-\tau))\d W_2(t),\ \ t>0,~~\ Y_0=\eta\in\C.
\end{equation}
Under ({\bf A2)}, \eqref{eq4} has a unique strong solution
$(Y(t))_{t\ge-\tau}$ (see, e.g., \cite[Theorem 2.2, P.150]{M08}). To
highlight  the initial value $\eta\in\C$ and the frozen segment
$\zeta\in\C$, we write the corresponding solution process
$(Y^\zeta(t,\eta))_{t\ge-\tau}$ and the segment process
$(Y^\zeta_t(\eta))_{t\ge0}$ instead of $(Y(t))_{t\ge-\tau}$ and
$(Y_t)_{t\ge0}$, respectively.

\smallskip

Our main result in this section is stated as below. It is concerned
with  ergodicity of the frozen SDE with memory.

\begin{thm}\label{Ergodicity}
{\rm Under ({\bf A2}) and ({\bf A3}),   $Y_t^\zeta(\eta)$ has a
unique invariant measure $\mu^\zeta$, and there exists  $\ll>0$ such
that
\begin{equation}\label{d6}
|\E b_1(\zeta,Y^\zeta_t(\eta))-\bar b_1(\zeta)|\1\e^{-\ll
t}(1+\|\eta\|_\8+\|\zeta\|_\8),\ \ \ t\ge0,~~\eta\in\C,
\end{equation}
  where
\begin{equation}\label{u1}
\bar{b}_1(\zeta):=\int_\C b_1(\zeta,\varphi)\mu^\zeta(\d\varphi),\ \
\ \zeta\in\C.
\end{equation}

 }
\end{thm}

\begin{proof}
The main idea of the proof concerning existence of an invariant
measure goes back to \cite[Lemma 2.4]{BWY}, which, nevertheless,
involves functional SDEs with additive noises.

\smallskip

Let $\scr P(\C)$ be the set of all probability measures on $\C$.
$W_2$ denotes the $L^2$-Wasserstein distance on $\scr P(\C)$ induced
by the bounded distance $\rr(\xi,\eta):= 1\land
\|\xi-\eta\|_\infty,$ i.e.,
$$W_2(\mu_1,\mu_2)= \inf_{\pi\in \C(\mu_1,\mu_2)} \big(\pi(\rr^2)\big)^{\ff 1 2},\ \ \mu_1,\mu_2\in \scr P(\C),$$
where $\C(\mu_1,\mu_2)$ is the set of all coupling  probability
measures with marginals  $\mu_1$ and $\mu_2$. It is well known that
$\scr P(\C)$ is a complete metric space w.r.t. the distance $W_2$
(
\cite[Lemma 5.3, P.174]{Chen} and \cite[Theorem 5.4,
P.175]{Chen}), and the convergence in $W_2$ is equivalent to the
weak convergence (
\cite[Theorem 5.6, P.179]{Chen}). Let
$P_t^{\zeta,\eta}$ be the law of the segment process
$Y_t^\zeta(\eta)$. According to the Krylov-Bogoliubov existence
theorem
(\cite[Theorem 3.1.1, P.21]{DZ}), if
$P_t^{\zeta,\eta}$ converges weakly to a probability measure
$\mu_{\eta}^\zeta$, then $\mu_{\eta}^\zeta$ is an invariant measure.
So,  in 
light of the previous discussion,     it suffices to
prove the assertions below:
 \beg{enumerate}
\item[(i)] $\{P_t^{\zeta,\eta}\}_{t\ge 0}$ is a
Cauchy sequence w.r.t. the distance $W_2$. If so, by  the
completeness of $\scr P(\C)$ w.r.t. the distance $W_2$, there is
$\mu_\eta^\zeta\in \scr P(\C)$ such that $\lim_{t\to\infty}
W_2(P_t^{\zeta,\eta},\mu_\eta^\zeta)=0$;
\item[(ii)] $W_2(\mu_\eta^\zeta,\mu_{\eta^\prime}^\zeta)=0$ for any
$\eta,\eta^\prime\in\C$ and  frozen $\zeta\in\C$, that is,
$\mu_\eta^\zeta$ is independent of $\eta$.
  \end{enumerate}
In the sequel, we shall claim that (i) and (ii) hold, respectively.
For any $t_2>t_1>\tau$ and the  frozen segment $\zeta\in\C$,
consider the following SDE with memory \beg{equation}\label{b10}\d
\bar Y(t)=b_2(\zeta,\bar Y(t),\bar Y(t-\tau))\d t+\si_2(\zeta,\bar
Y(t),\bar Y(t-\tau))\d W_2(t),\ t\in [t_2-t_1,t_2] \end{equation}
with the initial value $\bar Y_{t_2-t_1}=\eta.$ The solution process
and the segment process associated with \eqref{b10} are denoted by
$(\bar Y^\zeta(t,\eta))$ and $(Y^\zeta_t(\eta))$, respectively.
Observe that the laws of $Y^\zeta_{t_2}(\eta)$ and $\bar
Y^\zeta_{t_2}(\eta)$ are $P_{t_2}^{\zeta,\eta}$ and
$P_{t_1}^{\zeta,\eta}$, respectively.

\smallskip

By ({\bf A2}), there exists an $\aa>0$ such that
\begin{equation}\label{y1}
\|\si_2(\chi,x,y)-\si_2(\chi,x^\prime,y^\prime)\|\le
\aa(|x-x^\prime|+|y-y^\prime|),
\end{equation}
and
\begin{equation}\label{y2}
|b_2(\chi,0,0)|+\|\si_2(\chi,0,0)\|\le \aa(1+\|\chi\|_\8)
\end{equation}
for any $\chi\in\C$ and $x,x^\prime,y,y^\prime\in\R^n$. Accordingly,
\eqref{y1} and \eqref{y2}, together with ({\bf A3}),   yield that
there exist $\ll_1^\prime>\ll_2^\prime>0$, independent of $\chi,$
such that
\begin{equation}\label{b1}
\begin{split}
&2\<x,b_2(\chi,x,y)\>+\|\si_2(\chi,x,y)\|^2
\le-\ll_1^\prime|x|^2+\ll_2^\prime|y|^2+c(1+\|\chi|_\8^2)
\end{split}
\end{equation}
for any $\chi\in\C$ and $x,y\in\R^n$.   For a sufficiently small
$\ll^\prime>0$
 obeying
$\ll_1^\prime-\ll^\prime-\ll_2^\prime\e^{\ll^\prime\tau}=0$ due to
$\ll_1^\prime>\ll_2^\prime>0$, applying  It\^o's formula, we infer
from \eqref{b1} that
\begin{equation*}
\begin{split}
\e^{\ll^\prime
t}\E|Y^\zeta(t,\eta)|^2&\le|\eta(0)|^2+\int_0^t\e^{\ll^\prime
s}\E\{c(1+\|\zeta\|_\8^2)+\ll^\prime|Y^\zeta(s,\eta)|^2\\
&\quad-\ll_1^\prime|Y^\zeta(s,\eta)|^2+\ll_2^\prime|Y^\zeta(s-\tau,\eta)|^2\}\d
s\\
&\1\|\eta\|_\8^2+\e^{\ll^\prime t}(1+\|\zeta\|_\8^2),~~~t>0.
\end{split}
\end{equation*}
Consequently, we
arrive at
\begin{equation}\label{b2}
\E|Y^\zeta(t,\eta)|^2\1\e^{-\ll^\prime
t}\|\eta\|_\8^2+1+\|\zeta\|_\8^2,~~~t>0.
\end{equation}
Also, by the It\^o  formula, in addition to  the
Burkhold-Davis-Gundy (B-D-G for abbreviation) inequality, we derive
from ({\bf A3}), and  \eqref{y1}-\eqref{b2} that, for any
$t\ge\tau,$
\begin{equation}\label{b4}
\begin{split}
&\E\|Y^\zeta_t(\eta)\|_\8^2\\
&\11+\|\zeta\|_\8^2+\E|Y^\zeta(t-\tau,\eta)|^2+\int_{t-2\tau}^t\E|Y^\zeta(s,\eta)|^2\d
s\\
&\quad+2\E\Big(\sup_{t-\tau\le s\le
t}\Big|\int_{t-\tau}^s\<Y^\zeta(s,\eta),\si_2(\zeta,Y^\zeta(s,\eta),Y^\zeta(s-\tau,\eta))\d
W_2(s)\>\Big|\Big)\\
&\le\ff{1}{2}\E\|Y^\zeta_t(\eta)\|_\8^2+c\Big\{1+\|\zeta\|_\8^2+\E|Y^\zeta(t-\tau,\eta)|^2+\int_{t-2\tau}^t\E|Y^\zeta(s,\eta)|^2\d
s\Big\}.
\end{split}
\end{equation}
On the other hand, following 
the argument leading to 
\eqref{b4}, one
has
\begin{equation}\label{b0}
\begin{split}
\E\|Y^\zeta_t(\eta)\|_\8^2
&\le\ff{1}{2}\E\|Y^\zeta_t(\eta)\|_\8^2+c\Big\{1+\|\zeta\|_\8^2+\|\eta\|^2_\8+\int_0^t\E|Y^\zeta(s,\eta)|^2\d
s\Big\},~t\in[0,\tau].
\end{split}
\end{equation}
Thus, combining \eqref{b2} with \eqref{b4} and \eqref{b0} leads to
\begin{equation}\label{b5}
\E\|Y^\zeta_t(\eta)\|_\8^2\le c(\e^{-\ll^\prime
t}\|\eta\|_\8^2+1+\|\zeta\|_\8^2).
\end{equation}

\smallskip

In what follows, we assume  $t\in[t_2-t_1,t_2]$, and set
$\Gamma^\zeta(t,\eta):=Y^\zeta(t,\eta)-\bar Y^\zeta(t,\eta)$ for the
sake of notational simplicity.  Again, for a sufficiently small
$\ll>0$ such that $\ll_1-\ll-\ll_2\e^{\ll\tau}=0$ owing to
$\ll_1>\ll_2$, by the It\^o formula, it follows from ({\bf A2}) that
\begin{equation*}
\begin{split}
\e^{\ll
t}\E|\Gamma^\zeta(t,\eta)|^2
&\le\e^{\ll(t_2-t_1)}\E|\Gamma^\zeta(t_2-t_1,\eta)|^2\\
&\quad+\int_{t_2-t_1}^t\e^{\ll
s}\E\{(\ll-\ll_1)|\Gamma^\zeta(s,\eta)|^2+\ll_2|\Gamma^\zeta(s-\tau,\eta)|^2\}\d
s\\
&\le\e^{\ll(t_2-t_1)}\E|\Gamma^\zeta(t_2-t_1,\eta)|^2+\e^{\ll\tau}\int_{t_2-t_1-\tau}^{t_2-t_1}\e^{\ll
s}\E|\Gamma^\zeta(s,\eta)|^2\d s\\
&\1\e^{\ll
(t_2-t_1)}\|\eta\|^2_\8+\e^{\ll(t_2-t_1)}\E\|Y^\zeta_{t_2-t_1}(\eta)\|^2_\8.
\end{split}
\end{equation*}
This, together with \eqref{b5}, yields that
\begin{equation}\label{b6}
\E|\Gamma^\zeta(t,\eta)|^2\1\e^{-\ll
(t+t_1-t_2)}(1+\|\eta\|_\8^2+\|\zeta\|_\8^2).
\end{equation}
Imitating a similar procedure to derive \eqref{b4}, in particular,
we obtain from ({\bf A2}), \eqref{y1}, and \eqref{b6} that
\begin{equation}\label{b7}
\E\|\Gamma^\zeta_{t_2}(\eta)\|^2_\8\1\e^{-\ll
t_1}(1+\|\eta\|_\8^2+\|\zeta\|_\8^2).
\end{equation}
This further implies that
$$W_2(P_{t_1}^{\zeta,\eta},P_{t_2}^{\zeta,\eta})\le \E \{1\land \|Y^\zeta_{t_2}(\eta)-\bar Y^\zeta_{t_2}(\eta)\|_\infty^2\} \1\e^{-\ll
t_1}(1+\|\eta\|_\8^2+\|\zeta\|_\8^2),$$ which goes to zero as $t_1$
(hence $t_2$) tends to $\8$. Thus 
claim (i) holds.

\smallskip

By carrying out a similar argument to obtain \eqref{b7}, one finds
that
\begin{equation}\label{b8}
\E\|Y^\zeta_t(\eta)-Y^\zeta_t(\eta^\prime)\|_\8^2\1\e^{-\ll
t}\|\eta-\eta^\prime\|_\8^2.
\end{equation}
For fixed $\zeta\in\C$ and arbitrary $\eta,\eta\in\C$, observe that
\begin{equation}\label{b9}
W_2(\mu_\eta^\zeta,\mu_{\eta^\prime}^\zeta)\le
W_2(P_t^{\zeta,\eta},\mu_\eta^\zeta)+W_2(P_t^{\zeta,\eta^\prime},\mu_{\eta^\prime}^\zeta)+W_2(P_t^{\zeta,\eta},P_t^{\zeta,\eta^\prime}).
\end{equation}
Consequently, 
claim (ii) follows by taking \eqref{b8} and
\eqref{b9} into consideration.

\smallskip

By virtue of   \eqref{b5} and the invariance of $\mu^\zeta$,  it
then follows that
\begin{equation*}
\begin{split}
\int_\C\|\psi\|_\8^2\pi^\zeta(\d \psi)
&\le c\Big\{1+\|\zeta\|_\8^2+\e^{-\ll
t}\int_\C\|\psi\|_\8^2\pi^\zeta(\d \psi)\Big\}.
\end{split}
\end{equation*}
Thus, choosing $t>0$ sufficiently large such that $\dd:=c\e^{-\ll
t}<1$, one finds that
\begin{equation}\label{r3}
\int_\C\|\psi\|_\8^2\pi^\zeta(\d \psi)\11+\|\zeta\|_\8^2.
\end{equation}
 Next, with the aid of the invariance of $\pi^\zeta$, \eqref{b8},
and \eqref{r3},  we deduce from ({\bf A1}) that
\begin{equation*}
\begin{split}
|\E b_1(\zeta,Y^\zeta_t(\eta))-\bar b_1(\zeta)|
&\1\int_\C\E \|Y^\zeta_t(\eta)-Y^\zeta_t(\psi)\|_\8\pi^\zeta(\d\psi)
\1\e^{-\ff{\ll t}{2}}\int_\C
\|\eta-\psi\|_\8\pi^\zeta(\d\psi)\\
&\1\e^{-\ff{\ll t}{2}}(1+\|\eta\|_\8+\|\zeta\|_\8).
\end{split}
\end{equation*}
As a result,  \eqref{d6} follows.
\end{proof}

\begin{rem}
{\rm It should be noted that  
there are  
other alternative approaches to 
obtain
existence and uniqueness of invariant measures for functional SDEs.
Regarding to existence of invariant measures,  
Es-Sarhir et al. \cite{ESV}, and
Kinnally-Williams \cite{KW} by Arzel\`{a}--Ascoli's tightness
characterization, Bao et al. \cite{BYY14} using a remote start
method, Bao et al. \cite{BYY13} adopting Kurtz's Tightness
Criterion, and Rei$\beta$ et al. \cite{RR} by considering the
semi-martingale characteristics. As for uniqueness of invariant
measures, we refer to Hairer et al. \cite{HMS}, and
Kinnally-Williams \cite{KW}  by utilizing an asymptotic coupling
method. }
\end{rem}

The next corollary, which plays a crucial role in discussing strong
limit theorem for the averaging principle, states that $\bar b_1$,
defined by \eqref{u1}, enjoys a Lipschitz property.

\begin{cor}\label{bounded}
{\rm Under ({\bf A1})-({\bf A3}),  $\bar b_1:\C\mapsto\R^n$, defined
as in \eqref{u1}, is Lipschitz. }
\end{cor}

\begin{proof}
For arbitrary  $\phi,\zeta\in\C$, let
\begin{equation*}
\nn_\phi\bar b_1(\zeta)=\ff{\d}{\d\vv}\bar
b_1(\zeta+\vv\phi)\Big|_{\vv=0}
\end{equation*}
be the direction derivative of $\bar b_1$ at $\zeta$ along the
direction $\phi$. By Theorem \ref{Ergodicity},   we have
\begin{equation*}
\begin{split}
\nn_\phi\bar b_1(\zeta)&=\lim_{t\to\8}\E \nn_\phi
b_1(\zeta,Y^\zeta_t(\eta))\\
&=\lim_{t\to\8}\E\Big\{(\nn_\phi^{(1)}
b_1)(\zeta,Y^\zeta_t(\eta))+\Big(\nn_{\nn_\phi
Y^\zeta_t(\eta)}^{(2)}
b_1\Big)(\zeta,Y^\zeta_t(\eta))\Big\},~~\phi,\zeta,\eta\in\C.
\end{split}
\end{equation*}
According to ({\bf A1}), to verify that $\bar b_1:\C\mapsto\R^n$ is
Lipschitz, it remains to verify
\begin{equation}\label{a5}
\sup_{t\ge0}\E\|\nn_\phi Y^\zeta_t(\eta)\|_\8^2<\8.
\end{equation}
Observe that
 $\nn_\phi Y^{\zeta}(t, \eta)$   satisfies
the   following linear SDE with memory
\begin{equation*}
\begin{split}
\d (\nn_\phi Y^{\zeta}(t, \eta))&=\Big\{
(\nn_\phi^{(1)} b_{2})(\zeta, Y^\zeta(t,\eta),Y^\zeta(t-\tau,\eta))\\
&\quad+\Big(\nn_{\nn_\phi
Y^\zeta(t,\eta)}^{(2)}b_{2}\Big)(\zeta,Y^\zeta(t,\eta),Y^\zeta(t-\tau,\eta))\\
&\quad+
\Big(\nn_{\nn_\phi Y^\zeta(t-\tau,\eta)}^{(3)}b_{2}\Big)(\zeta,Y^\zeta(t,\eta),Y^\zeta(t-\tau,\eta))\Big\} \d t\\
&\quad+\Big\{(\nn_\phi^{(1)}\si_{2})(\zeta,
Y^\zeta(t,\eta),Y^\zeta(t-\tau,\eta))\\
&\quad+\Big(\nn_{\nn_\phi Y^\zeta(t,\eta)}^{(2)}\si_{2}\Big)(\zeta,
Y^\zeta(t,\eta),Y^\zeta(t-\tau,\eta))\\
&\quad+\Big(\nn_{\nn_\phi
Y^\zeta(t-\tau,\eta)}^{(3)}\si_{2}\Big)(\zeta,
Y^\zeta(t,\eta),Y^\zeta(t-\tau,\eta))\Big\} \d W_2(t),~~~t>0
\end{split}
\end{equation*}
with the initial 
data $\nn_\phi Y^\zeta_0(\eta)=0.$ In the sequel,
let  $\chi\in\C$ and $x,x^\prime,y,y^\prime\in\R^n$. For any
$\vv>0,$ it is trivial to see from
 ({\bf A3}) that
\begin{equation*}
\begin{split}
&2\vv\<x,b_2(\chi,x^\prime+\vv x,y^\prime+\vv y)-b_2(\chi,
x^\prime,y^\prime)\>+\|\si_2(\chi,x^\prime+\vv x,y^\prime+\vv y)-\si_2(\chi,x^\prime,y^\prime)\|^2\\
&\qquad\le-\ll_1\vv^2|x|^2+\ll_2\vv^2|y|^2.
\end{split}
\end{equation*}
Multiplying $\vv^{-2}$ on both sides, followed by  
sending
$\vv\downarrow0$, gives that
\begin{equation}\label{c1}
\begin{split}
&2\<x,(\nn_x^{(2)} b_2)(\chi,x^\prime,y^\prime)+(\nn_y^{(3)}
b_2)(\chi,x^\prime,y^\prime)\>\\
&\qquad\quad +\|(\nn_x^{(2)} \si_2)(\chi,x^\prime,y^\prime)+(\nn_y^{(3)}
\si_2)(\chi,x^\prime,y^\prime)\|^2\\
& \quad \le-\ll_1|x|^2+\ll_2|y|^2.
\end{split}
\end{equation}
On the other hand, by virtue of \eqref{y1}, for any $\vv>0$, one has
\begin{equation*}
\|\si_2(\chi,x^\prime+\vv x,y^\prime+\vv
y)-\si_2(\chi,x^\prime,y^\prime)\|^2\le\aa\vv^2(|x|^2+|y|^2),
\end{equation*}
which further yields  that
\begin{equation}\label{c2}
\|(\nn_x^{(2)} \si_2)(\chi,x^\prime,y^\prime)+(\nn_y^{(3)}
\si_2)(\chi,x^\prime,y^\prime)\|^2\le\aa(|x|^2+|y|^2).
\end{equation}
Thus, with \eqref{c1} and \eqref{c2} in hand,   \eqref{a5} holds by
repeating 
the argument which  
\eqref{b5} is obtained.
\end{proof}

\section{Preliminary Results}
\label{sec:lem}
In this paper, we 
study the strong deviation between
the slow component $X^\vv(t)$ and the averaged component $\bar
X(t)$, which satisfies the following functional SDE
\begin{equation}\label{3eq3} \d \bar X(t)=\bar{b}_1(\bar X_t)\d
t+\si_1(\bar X_t)\d W_1(t),\ \ \ \bar X_0=\xi\in\C,
\end{equation}
where $\bar b_1:\C\mapsto\R^n$ is defined as in \eqref{u1}. To
achieve this goal, we need to  construct some auxiliary
two-time-scale stochastic systems with memory and provide a number
of preliminary lemmas.

\smallskip

Throughout this paper, we fix $T>0$ and set
$\dd:=\ff{\tau}{N}\in(0,1)$  for a positive integer $N$ sufficiently
large.
For any $t\in[0,T],$ consider the following   auxiliary
two-time-scale systems of functional SDEs
\begin{equation}\label{eq13}
\d \tt X^\vv(t)=b_1( X_{t_\dd}^\vv,\tt Y_t^\vv)\d t+\si_1(
X_{t_\dd}^\vv)\d W_1(t), ~~X_0^\vv=\xi\in\C,
\end{equation}
and
\begin{equation}\label{eq14}
\begin{cases}
\d \tt Y^\vv(t)=\ff{1}{\vv}b_2( X^\vv_{t_\dd},  \tt Y^\vv(t),\tt
Y^\vv(t-\tau))\d t+\ff{1}{\ss{\vv}}\si_2( X^\vv_{t_\dd},  \tt Y^\vv(t),\tt Y^\vv(t-\tau))\d W_2(t), \\
\tt Y^\vv(t_\dd)=Y^\vv(t_\dd)
\end{cases}
\end{equation}
with the initial value $\tt Y^\vv_0=Y^\vv_0=\eta\in\C$, where
$t_\dd:=\lf t/\dd\rf\dd$, the nearest breakpoint preceding $t,$ with
$\lf t/\dd\rf$ being the  integer part of $t/\dd$.

\smallskip

To proceed, we present several preliminary lemmas. The first lemma
concerns the continuity in the mean $L^p$-norm sense for the
displacement of the segment process $(X_t^\vv)_{t\in[0,T]}$.

\begin{lem}\label{L3}
{\rm Under  ({\bf A1}) and ({\bf A4}),
\begin{equation*}
\sup_{t\in[0,T]}\E\|X^\vv_t-X^\vv_{t_\dd}\|_\8^p\1_T\dd^{\ff{p-2}{2}},~~~p>2.
\end{equation*}
}
\end{lem}

\begin{proof}
In accordance with
\cite[Theorem 4.1, P.160]{M08}, we have
\begin{equation}\label{eq15}
\E\Big(\sup_{0\le t\le T}\|X^\vv_t\|^p_\8\Big)\1_T1+\|\xi\|_\8^p.
\end{equation}
 Observe that
\begin{equation*}
\begin{split}
\E\|X^\vv_t-X^\vv_{t_\dd}\|_\8^p
&\le\sum_{m=0}^{N-1}\E\Big(\sup_{-(m+1)\dd\le\theta\le-m\dd}|X^\vv(t+\theta)-X^\vv(t_\dd+\theta)|^p\Big)\\
&=:\sum_{m=0}^{N-1}J_p(t,m,\dd),
\end{split}
\end{equation*}
where $N=\tau/\dd$ by the definition of $\dd.$
To 
complete  the proof of Lemma \ref{L3}, it is sufficient to show
\begin{equation}\label{s1}
J_p(t,m,\dd)\1_T\dd^{\ff{p}{2}}.
\end{equation}
For any $t\in[0,T]$, take $k\ge0$ such that $t\in[k\dd,(k+1)\dd).$
Thus, for any $\theta\in[-(m+1)\dd,-m\dd]$, one has
\begin{equation*}
t+\theta\in[(k-m-1)\dd,(k+1-m)\dd] ~\mbox{ and
}~t_\dd+\theta\in[(k-m-1)\dd,(k-m)\dd].
\end{equation*}
In what follows, we 
consider three cases.

\smallskip

\noindent{{\bf Case 1:} $m\le k-1.$} Invoking  H\"older's inequality
and B-D-G's inequality, we obtain from ({\bf A1}) and \eqref{eq15}
that
\begin{equation}\label{e1}
\begin{split}
&J_p(t,m,\dd)\\
&\1\dd^{p-1}\int_{(k-m-1)\dd}^{t-m\dd}\E|b_1(X_s^\vv,Y_s^\vv)|^p\d
s+\E\Big(\sup_{-(m+1)\dd\le\theta\le-m\dd}\Big|\int_{k\dd+\theta}^{t+\theta}\si_1(X_s^\vv)\d
W_1(s)\Big|^p\Big)\\
&\1\dd^{p-1}\int_{(k-m-1)\dd}^{t-m\dd}\E|b_1(X_s^\vv,Y_s^\vv)|^p\d
s+\E\Big(\Big|\int_{(k-m-1)\dd}^{t-(m+1)\dd}\si_1(X_s^\vv)\d
W_1(s)\Big|^p\Big)\\
&\quad+\E\Big(\sup_{-(m+1)\dd\le\theta\le-m\dd}\Big|\int_{t-(m+1)\dd}^{t+\theta}\si_1(X_s^\vv)\d
W_1(s)\Big|^p\Big)\\
&\quad+\E\Big(\sup_{-(m+1)\dd\le\theta\le-m\dd}\Big|\int_{(k-m-1)\dd}^{k\dd+\theta}\si_1(X_s^\vv)\d
W_1(s)\Big|^p\Big)\\
&\1\dd^{p-1}\int_{(k-m-1)\dd}^{t-m\dd}\E|b_1(X_s^\vv,Y_s^\vv)|^p\d
s+\dd^{\ff{p-2}{2}}\E\Big(\int_{(k-m-1)\dd}^{t-(m+1)\dd}\|\si_1(X_s^\vv)\|^p\d
s\Big)\\
&\quad+\E\Big(\int_{t-(m+1)\dd}^{t-m\dd}\|\si_1(X_s^\vv)\|^2\d
s\Big)^{p/2}+\E\Big(\int_{(k-m-1)\dd}^{(k-m)\dd}\|\si_1(X_s^\vv)\|^2\d
s\Big)^{p/2} \\
&\1_T\dd^{\ff{p}{2}}.
\end{split}
\end{equation}

\smallskip

\noindent \noindent{{\bf Case 2:} $m\ge k+1.$} In view of ({\bf
A5}), it follows that
\begin{equation*}
|X^\vv(t+\theta)-X^\vv(t_\dd+\theta)|^p=|\xi(t+\theta)-\xi(t_\dd+\theta)|^p\1\dd^p.
\end{equation*}

\smallskip
\noindent{\bf Case 3:} $m=k$.  Also, by H\"older's inequality and
B-D-G's inequality,   we deduce from ({\bf A1}) and
 \eqref{eq15}  that
\begin{equation}\label{c3}
\begin{split}
J_p(t,m,\dd)&=\E\Big(\sup_{-(k+1)\dd\le\theta\le-k\dd}|X^\vv(t+\theta)-X^\vv(k\dd+\theta)|^p\Big)\\
&\1\dd^p+{\E}\Big(\sup_{-(k+1)\dd\le\theta\le-k\dd}(|X^\vv(t+\theta)-X^\vv(0)|^p{\bm{1}}_{\{t+\theta>0\}})\Big)\\
&\1
\dd^p+\E\Big(\sup_{-t\le\theta\le-k\dd}\Big|\int_0^{t+\theta}b_1(X_s^\vv,Y_s^\vv)\d
s\Big|^p\Big)\\
&\quad+\E\Big(\sup_{-t\le\theta\le-k\dd}\Big|\int_0^{t+\theta}\si_1(X_s^\vv)\d
W_1(s)\Big|^p\Big)\\
&\1 \dd^p+\dd^{p-1}\int_0^{t-k\dd}\E|b_1(X_s^\vv,Y_s^\vv)|^p\d
s+\dd^{\ff{p-2}{2}}\int_0^{t-k\dd}\E\|\si_1(X_s^\vv)\|^p\d
s\\
&\1_T\dd^{\ff{p}{2}},
\end{split}
\end{equation}
where $a^+:=\max\{a,0\}$ for $a\in\R.$ Consequently, the desired
assertion \eqref{s1} is   finished by taking the discussions above
into account.
\end{proof}

The lemma below provides an error bound of the difference in the
strong sense between the slow component $(X^\vv(t))$ and its
approximation $(\tt X^\vv(t))$.

\begin{lem}\label{L4}
{\rm Assume that  ({\bf A1}) and ({\bf A2}) hold and suppose further
$\vv/\dd\in(0,1)$. Then, there exists $\bb>0$ such that
\begin{equation*}
\E\Big(\sup_{0\le s\le T}|X^\vv(t)-\tt
X^\vv(t)|^p\Big)\1_T\dd^{\ff{p-2}{2}}(1+\vv^{-1}\e^{\ff{\bb\dd}{\vv}}),~~~p>2.
\end{equation*}
 }
\end{lem}

\begin{proof}
In view of H\"older's inequality and B-D-G's inequality, it follows
from ({\bf A1}) and Lemma \ref{L3} that
\begin{equation*}
\begin{split}
\E\Big(\sup_{0\le s\le t}|X^\vv(s)-\tt X^\vv(s)|^p\Big)
&\1_T\int_0^t\E\{\|X_s^\vv- X_{s_\dd}^\vv\|_\8^p+\|Y^\vv_s-\tt Y^\vv_s\|^p_\8\}\d s\\
&\1_T \dd^{\ff{p-2}{2}}+\int_0^t\E\|Y^\vv_s-\tt Y^\vv_s\|^p_\8\d s,\
\ \ t\in(0,T].
\end{split}
\end{equation*}
 Therefore,  to finish the argument of Lemma \ref{L4}, it   suffices  to
 show that there exists $\bb>0$ such that
\begin{equation}\label{e3}
\sup_{t\in[0,T]}\E\|Y^\vv_t-\tt
Y^\vv_t\|_\8^p\1_T\vv^{-1}\dd^{\ff{p-2}{2}} \e^{\ff{\bb\dd}{\vv}}.
\end{equation}
In what follows, we verify
claim \eqref{e3} by an induction argument.
For any $t\in[0,\tau)$, due to $Y^\vv_0=\tt Y^\vv_0=\eta$,  it is
readily to check that
\begin{equation*}
\begin{split}
\E\|Y_t^\vv-\tt Y_t^\vv\|_\8^p
\le\sum_{j=0}^{\lf t/\dd\rf}\E\Big(\sup_{j\dd\le s\le
((j+1)\dd)\wedge t}|Y^\vv(s)-\tt Y^\vv(s)|^p\Big)=: I(t,\dd).
\end{split}
\end{equation*}
By means of It\^o's formula and B-D-G's inequality, together with
$\tt Y^\vv(t_\dd)=Y^\vv(t_\dd)$, we obtain from ({\bf A2})  that
\begin{equation*}
\begin{split}
\E&\Big(\sup_{j\dd\le s\le ((j+1)\dd)\wedge t}|Y^\vv(s)-\tt Y^\vv(s)|^p\Big)\\
&\quad\le\ff{c}{\vv}\int_{j\dd}^{((j+1)\dd)\wedge
t}\{\E\|X_s^\vv-X_{s_\dd}^\vv\|_\8^2+\E|Y^\vv(s)-\tt Y^\vv(s)|^p
\}\d
s\\
&\qquad \ +\ff{1}{2}\E\Big(\sup_{j\dd\le s\le((j+1)\dd)\wedge
t}|Y^\vv(s)-\tt
 Y^\vv(s)|^p\Big),~~~t\in[0,\tau].
\end{split}
\end{equation*}
Consequently, we conclude that
\begin{equation}\label{e2}
\begin{split}
I(t,\dd)
&\1\ff{1}{\vv}\int_0^t\E\|X_s^\vv-X_{s_\dd}^\vv\|_\8^2\d
s+\ff{1}{\vv}\int_0^\dd\sum_{j=0}^{\lf t/\dd\rf}\E\Big(\sup_{j\dd\le
r\le ((j\dd+s))\wedge t}|Y^\vv(r)-\tt Y^\vv(r)|^p\Big) \d s\\
&\1\ff{1}{\vv}\int_0^t\E\|X_s^\vv-X_{s_\dd}^\vv\|_\8^2\d
s+\ff{1}{\vv}\int_0^\dd I(t,s) \d s.
\end{split}
\end{equation}
This, combining Lemma \ref{L3} with   Gronwall's  inequality, gives
that
\begin{equation}\label{d7}
\E\|Y_t^\vv-\tt Y_t^\vv\|_\8^p\1
\vv^{-1}\dd^{\ff{p-2}{2}}\e^{\ff{c\dd}{\vv}},\ \ \ \ \ t\in[0,\tau)
\end{equation}
for some $c>0.$ Next, for any $t\in[\tau,2\tau)$,   thanks to
\eqref{d7}, it is immediate to note that
\begin{equation*}
\begin{split}
\E\|Y_t^\vv-\tt Y_t^\vv\|_\8^p
&\le\E\Big(\|Y^\vv_\tau-\tt
Y^\vv_\tau\|^p_\8\Big)+\E\Big(\sup_{\tau\le s\le t}|Y^\vv(s)-\tt
Y^\vv(s)|^p\Big)\\
&\le c\Big\{
\vv^{-1}\dd^{\ff{p-2}{2}}\e^{\ff{c\dd}{\vv}}+\sum_{j=0}^{\lf
t-\tau\rf}\E\Big(\sup_{(N+j)\dd\le s\le ((N+j+1)\dd)\wedge
t}|Y^\vv(s)-\tt
Y^\vv(s)|^p\Big)\Big\}\\
&=:c\{\vv^{-1}\dd^{\ff{p-2}{2}}\e^{\ff{c\dd}{\vv}}+M(t,\tau,\dd)\}.
\end{split}
\end{equation*}
Carrying out a similar argument to derive \eqref{e2}, we deduce from
\eqref{d7} that
\begin{equation*}
\begin{split}
M(t,\tau,\dd)&\1\ff{1}{\vv}\int_{\tau}^t\E\|X_s^\vv-X_{s_\dd}^\vv\|_\8^2\d s\\
&\quad+\ff{1}{\vv}\int_0^\dd\sum_{j=0}^{\lf
t-\tau\rf}\E\Big(\sup_{(N+j)\dd\le r\le ((N+j)\dd+s)\wedge
t}|Y^\vv(r)-\tt Y^\vv(r)|^p\Big)\d s\\
&\quad+\ff{1}{\vv}\int_0^{\dd}\sum_{j=0}^{\lf
t-\tau\rf}\E\Big(\sup_{j\dd\le s\le((j+1)\dd)\wedge
(t-\tau)}|Y^\vv(s)-\tt
Y^\vv(s)|^p\Big) \d s\\
&\1\ff{\dd^{\ff{p-2}{2}}}{\vv}+\ff{\dd}{\vv}\cdot\ff{\dd^{\ff{p-2}{2}}}{\vv}\e^{\ff{c\dd}{\vv}}+\ff{1}{\vv}\int_0^\dd
M(t,\tau,s)\d s.
\end{split}
\end{equation*}
Thus,  the  Gronwall inequality reads
\begin{equation*}
\begin{split}
M(t,\tau,\dd)&\1\Big\{\ff{\dd^{\ff{p-2}{2}}}{\vv}+\ff{\dd}{\vv}\cdot\ff{\dd^{\ff{p-2}{2}}}{\vv}\e^{\ff{c\dd}{\vv}}\Big\}
\e^{\ff{c\dd}{\vv}}\1\ff{\dd}{\vv}\cdot\ff{\dd^{\ff{p-2}{2}}}{\vv}\e^{\ff{c\dd}{\vv}}\1\ff{\dd^{\ff{p-2}{2}}}{\vv}\e^{\ff{c\dd}{\vv}},
\end{split}
\end{equation*}
where we have  used $\vv/\dd\in(0,1)$ in the second step. Finally,
\eqref{e3} follows by repeating the previous procedure.
\end{proof}

The following consequence explores  a uniform estimate  w.r.t. the
parameter $\vv$ for the segment process   associated with the
auxiliary fast motion.

\begin{lem}\label{l5}
{\rm Assume that  ({\bf A1}) and ({\bf A3}) hold. Then, there exists
$C_T>0$, independent of $\vv$, such that
\begin{equation}\label{w2}
\sup_{t\in[0,T]}\E\|\tt Y^\vv_t\|_\8^2\le C_T.
\end{equation}
}
\end{lem}

\begin{proof}
From \eqref{eq2},  it follows that
\begin{equation}\label{w1}
\begin{split}
Y^\vv(t)&=\eta(0)+\int_0^{t/\vv}b_2(X^\vv_{\vv s},Y^\vv(\vv s), Y^\vv(\vv s-\tau))\d t\\
&\quad+\int_0^{t/\vv}\si_2(X^\vv_{\vv s}, Y^\vv(\vv s), Y^\vv(\vv
s-\tau))\d \bar W_2(s),~~~t>0,
\end{split}
\end{equation}
where we used the fact that $\bar W(t):=\ff{1}{\ss{\vv}}W_2(\vv t)$
is a Brownian motion. For fixed $\vv>0$ and  $t\ge0$, let $\bar
Y^\vv(t+\theta)=Y^\vv (\vv t+\theta),\theta\in[-\tau,0]. $ So, one
has $\bar Y^\vv_t=Y^\vv_{\vv t}.$ Observe that \eqref{w1} can be
rewritten as
\begin{equation*}
\begin{split}
\bar Y^\vv(t/\vv)&=\eta(0)+\int_0^{t/\vv}b_2(X^\vv_{\vv s}, \bar
Y^\vv( s), \bar Y^\vv( s-\tau ))\d s+\int_0^{t/\vv}\si_2(X^\vv_{\vv
s}, \bar Y^\vv( s), \bar Y^\vv( s-\tau))\d \bar W_2(s).
\end{split}
\end{equation*}
Then, following 
the argument
 to 
obtain
\eqref{b5}, for any $s>0$ we
can deduce that
\begin{equation*}
\E\|\bar Y^\vv_s\|_\8^2\11+\|\eta\|_\8^2\e^{-\ll
s}+\E\Big(\sup_{0\le r\le \vv s}\|X_r^\vv\|_\8^2\Big).
\end{equation*}
This, together with  $\bar Y^\vv_t=Y^\vv_{\vv t},$ gives that
\begin{equation*}
\E\| Y^\vv_{\vv s}\|_\8^2\11+\|\eta\|_\8^2\e^{-\ll
s}+\E\Big(\sup_{0\le r\le \vv s}\|X_r^\vv\|_\8^2\Big).
\end{equation*}
In particular, taking $s=t/\vv$ we arrive at
\begin{equation*}
\E\| Y^\vv_t\|_\8^2\11+\|\eta\|_\8^2+\E\Big(\sup_{0\le r\le
t}\|X_r^\vv\|_\8^2\Big).
\end{equation*}
This, together with \eqref{eq15}, yields that
\begin{equation*}
\sup_{t\in[0,T]}\E\| Y^\vv_t\|_\8^2\le C_T
\end{equation*}
for some $C_T>0.$ Observe from \eqref{e3} and H\"oder's inequality
that
\begin{equation*}
\begin{split}
\E\|\tt Y_t^\vv\|_\8^2&\le2\E\|Y_t^\vv-\tt Y_t^\vv\|_\8^2+2\E\|
Y_t^\vv\|_\8^2\\
&\1_T1+\Big(\vv^{-1}\dd^{\ff{p-2}{2}}
\e^{\ff{\bb\dd}{\vv}}\Big)^{2/p}, ~~~p>4.
\end{split}
\end{equation*}
Next, taking $\dd=\vv(-\ln\vv)^{\ff{1}{2}}$ in the estimate above
and letting $y=(-\ln\vv)^{\ff{1}{2}}$, we have
\begin{equation*}
\begin{split}
\E\|\tt Y_t^\vv\|_\8^2
&\1_T1+\Big(\e^{y^2}(\e^{-y^2}y)^{\ff{p-2}{2}} \e^{\bb
y}\Big)^{2/p}, ~~~p>4.
\end{split}
\end{equation*}
Then, the desired assertion follows since the leading term
$\e^{y^2}(\e^{-y^2}y)^{\ff{p-2}{2}} \e^{\bb y}\rightarrow0$ as
$y\uparrow\8$ whenever $p>4.$

\end{proof}

\section{ A Strong Limit Theorem for  the Slow Component }\label{sec:avg}

With several preliminary lemmas
at our hands,
we
are in position to present our main result.

\begin{thm}\label{main}
{\rm Under ({\bf A1})-({\bf A4}), one has
\begin{equation*}
\lim_{\vv\to0}\E\Big(\sup_{0\le t\le T}|X^\vv(t)-\bar
X(t)|^p\Big)=0,~~p>0.
\end{equation*}

}
\end{thm}

\begin{proof}
For  any $t\in[0,T]$ and $p>0$, set
\begin{equation*}
\Lambda(t):=\E\Big(\sup_{0\le s\le t}|X^\vv(s)-\bar X(s)|^p\Big)
~~~\mbox{ and }~~~\Gamma(t):=\E\Big(\sup_{0\le s\le t}|\tt
X^\vv(s)-\bar X(s)|^p\Big).
\end{equation*}
By  H\"older's inequality,  it is sufficient to verify that
\begin{equation}\label{s2}
\lim_{\vv\to0}\LL(T)=0,~~~p>4.
\end{equation}
In what follows, let $t\in[0,T]$ be arbitrary and assume $p>4$. For
any $t\in[0,T]$, it follows from Lemma \ref{L4} that
\begin{equation}\label{a6}
\begin{split}
\LL(t)&\1\E\Big(\sup_{0\le s\le t}|X^\vv(s)-\tt X^\vv(s)|^p\Big)+
\GG(t)\1\dd^{\ff{p-2}{2}}\Big(1+\ff{1}{\vv}\e^{\ff{\bb\dd}{\vv}}\Big)+
\GG(t).
\end{split}
\end{equation}
 Next, if we can  show that
\begin{equation}\label{s4}
\GG(t) \1
\dd^{\ff{p-2}{2}}\Big(1+\ff{1}{\vv}\e^{\ff{\bb\dd}{\vv}}\Big)+\Big(\ff{\vv}{\dd}\Big)^\nu+\int_0^t\LL(s)\d
s
\end{equation}
for some $\nu\in(0,1),$ inserting \eqref{s4} back into \eqref{a6}
and utilizing   Gronwall's inequality, we deduce that
\begin{equation*}
\LL(t)\1\dd^{\ff{p-2}{2}}\Big(1+\ff{1}{\vv}\e^{\ff{\bb\dd}{\vv}}\Big)+\Big(\ff{\vv}{\dd}\Big)^\nu.
\end{equation*}
Thus, the desired assertion \eqref{s2} follows by
choosing $\dd=\vv(-\ln\vv)^{\ff{1}{2}}$. Indeed, it is easy to see
that $\vv/\dd\in(0,1)$, which is prerequisite in Lemma \ref{L4}, for
$\vv\in(0,1)$ small enough, and that $\dd\rightarrow0$ as
$\vv\downarrow0$. Furthermore, let $y=(-\ln\vv)^{\ff{1}{2}}$ (hence
$\vv=\e^{-y^2}$), which goes into infinity as $\vv$ tends to zero.
Then, we have
\begin{equation*}
\begin{split}
\LL(t) &\1(\e^{-y^2}y)^{\ff{p-2}{2}}\Big(1+\e^{y^2+\bb
y}\Big)+y^{-\nu},
\end{split}
\end{equation*}
which goes to zero  by taking $p>4$ and letting $y\uparrow\8.$

\smallskip

Next, we intend to claim \eqref{s4}. Set
\begin{equation*}
\GG_p(t,\dd,\vv):=\E\Big(\sup_{0\le s\le t}\Big|\int_0^s\{b_1(X_{r_\dd}^\vv,\tt Y_r^\vv)-\bar{b}_1(X_{r_\dd}^\vv)\}\d
r\Big|^p\Big),~~t\in[0,T].
\end{equation*}
Applying  H\"older's inequality, B-D-G's inequality, Lipschitz
property of $\bar b_1$ due to Corollary \ref{bounded}, and Lemma
\ref{L3}, we derive that
\begin{equation*}
\begin{split}
\GG(t)&\1\E\Big(\sup_{0\le s\le
t}\Big|\int_0^t\{b_1(X_{s_\dd}^\vv,\tt Y_s^\vv)-\bar{b}_1(\bar
X_s)\}\d
s\Big|^p\Big)+\int_0^t\E\|\si_1(X_{s_\dd}^\vv)-\si_1(\bar X_s)\|^p\d s\\
 &\1\GG_p(t,\dd,\vv)+\int_0^t\E|\bar{b}_1(X_{s_\dd}^\vv)-\bar{b}_1(X_s^\vv)|^p\d
s+\int_0^t\E|\bar{b}_1(X_s^\vv)
-\bar{b}_1(\tt X_s^\vv)|^p\d s\\
&\quad+\int_0^t\E|\bar{b}_1(\tt X_s^\vv)-\bar{b}_1(\bar X_s)|^p\d s+\int_0^t\E\|\si_1( X_{s_\dd}^\vv)-\si_1(\bar X_s)\|^p\d s\\
&\1\GG_p(t,\dd,\vv)+\int_0^t\E\|X_s^\vv-\tt X_s^\vv\|_\8\d
s+\int_0^t\E\|X_{s_\dd}^\vv-X_s^\vv\|^p_\8\d
s+\int_0^t\GG(s)\d s+\int_0^t\LL(s)\d s\\
&\1\dd^{\ff{p-2}{2}}+\ff{1}{\vv}\dd^{\ff{p-2}{2}}\e^{\ff{c\dd}{\vv}}+\GG_p(t,\dd,\vv)+\int_0^t\GG(s)\d
s+\int_0^t\LL(s)\d s,
\end{split}
\end{equation*}
which, together with Gronwall's inequality, leads to
\begin{equation}\label{s7}
\GG(t) \1
\dd^{\ff{p-2}{2}}\Big(1+\ff{1}{\vv}\e^{\ff{\bb\dd}{\vv}}\Big)+\GG_p(t,\dd,\vv)+\int_0^t\LL(s)\d
s,
\end{equation}
where we have utilized the fact that $\GG_p(t,\dd,\vv)$ is
nondecreasing with respect to $t.$ By 
comparing \eqref{s4} with
\eqref{s7}, we need only 
prove
\begin{equation}\label{s8}
\GG_p(t,\dd,\vv)\1\Big(\ff{\vv}{\dd}\Big)^\nu
\end{equation}
for some $\nu\in(0,1).$

\smallskip
 Let
\begin{equation*}
\Upsilon_p(k,\dd,\vv)=\E\Big(\Big|\int_{k\dd}^{((k+1)\dd)\wedge
t}\{b_1(X_{k\dd}^\vv,\tt Y_s^\vv)-\bar{b}_1(X_{k\dd}^\vv)\}\d
s\Big|^p\Big) ~~\mbox{for any }p>0.
\end{equation*}
 In the sequel, we show that \eqref{s8} holds.
By   H\"older's inequality, we obtain that
\begin{equation}\label{s3}
\begin{split}
\GG_p(t,\dd,\vv)
 &=\E\Big(\sup_{0\le s\le t}\Big|\sum_{k=0}^{\lf
s/\dd\rf}\int_{k\dd}^{((k+1)\dd)\wedge t}\{b_1(X_{k\dd}^\vv,\tt
Y_r^\vv)-\bar{b}_1(X_{k\dd}^\vv)\}\d r\Big|^p\Big)\\
&\le\E\Big(\sup_{0\le s\le t}\Big((\lf
s/\dd\rf+1)^{p-1}\sum_{k=0}^{\lf
s/\dd\rf}\Upsilon_p(k,\dd,\vv)\Big)\Big)\\
&\le(\lf t/\dd\rf+1)^{p-1}\sum_{k=0}^{\lf
t/\dd\rf}\Upsilon_p(k,\dd,\vv)\\
&\le(\lf t/\dd\rf+1)^p\max_{0\le k\le\lf
t/\dd\rf}\Upsilon_p(k,\dd,\vv).
\end{split}
\end{equation}
For any $p^\prime\in(1,2)$, by H\"older's inequality, ({\bf A1}),
and \eqref{eq15},  observe that
\begin{equation*}
\begin{split}
\Upsilon_p(k,\dd,\vv)
&\le\Upsilon_2(k,\dd,\vv)^{\ff{p^\prime}{2}}\Big(\E\Big(\Big|\int_{k\dd}^{((k+1)\dd)\wedge
t}\{b_1(X_{k\dd}^\vv,\tt Y_s^\vv)-\bar{b}_1(X_{k\dd}^\vv)\}\d
s\Big|^{\ff{2(p-p^\prime)}{2-p^\prime}}\Big)\Big)^{\ff{2-p^\prime}{2}}\\
&\le\Upsilon_2(k,\dd,\vv)^{\ff{p^\prime}{2}}\Big(\dd^{\ff{2(p-p^\prime)}{2-p^\prime}-1}\E\Big(\Big|\int_{k\dd}^{((k+1)\dd)\wedge
t}|b_1(X_{k\dd}^\vv,\tt
Y_s^\vv)-\bar{b}_1(X_{k\dd}^\vv)|^{\ff{2(p-p^\prime)}{2-p^\prime}}\d
s\Big|\Big)\Big)^{\ff{2-p^\prime}{2}}\\
&\1
\Upsilon_2(k,\dd,\vv)^{\ff{p^\prime}{2}}\dd^{\ff{2(p-p^\prime)}{2-p^\prime}\times\ff{2-p^\prime}{2}}\\
&\1\Upsilon_2(k,\dd,\vv)^{\ff{p^\prime}{2}}\dd^{p-p^\prime},~~~~p>4.
\end{split}
\end{equation*}
Substituting this into \eqref{s3}, we arrive at
\begin{equation*}
\GG_p(t,\dd,\vv)\1\Upsilon_2(k,\dd,\vv)^{\ff{p^\prime}{2}}\dd^{-p^\prime}.
\end{equation*}
Thus, to complete the argument, it remains to show that
\begin{equation*}
\Upsilon_2(k,\dd,\vv)\1\vv\dd.
\end{equation*}
Also, by virtue of H\"older's inequality, ({\bf A1}), and
\eqref{eq15}, we derive  that
\begin{equation}\label{r1}
\begin{split}
&\Upsilon_2(k,\dd,\vv)\\
&=2\int_{k\dd}^{((k+1)\dd)\wedge t}\int_s^{((k+1)\dd)\wedge
t}\E\<b_1(X_{k\dd}^\vv,\tt
Y_s^\vv)-\bar{b}_1(X_{k\dd}^\vv),b_1(X_{k\dd}^\vv,\tt
Y_r^\vv)-\bar{b}_1(X_{k\dd}^\vv)\>\d r\d s\\
&\1\int_{k\dd}^{(k+1)\dd}\int_s^{(k+1)\dd}(\E|\E((b_1(X_{k\dd}^\vv,\tt
Y_r^\vv)-\bar{b}_1(X_{k\dd}^\vv))|\F_s)|^2)^{1/2}\d r\d s.
\end{split}
\end{equation}
For any $r\in[k\dd,(k+1)\dd)$, by the definition of $\tt Y^\vv$,
defined as in \eqref{eq14}, it follows that
\begin{equation}\label{q1}
\begin{split}
\tt Y^\vv(r)&=\tt Y^\vv(k\dd)+\ff{1}{\vv}\int_{k\dd}^rb_2( X^\vv_{k
\dd}, \tt
Y^\vv(u), \tt Y^\vv(u-\tau))\d u\\
&\quad+\ff{1}{\ss{\vv}}\int_{k\dd}^r\si_2(X^\vv_{k \dd},\tt
Y^\vv(u), \tt Y^\vv(u-\tau))\d
W_2(u)\\
&=\tt Y^\vv(k\dd)+\int_0^{\ff{r-k\dd}{\vv}}b_2(X^\vv_{k \dd},\tt
Y^\vv(k\dd+\vv u), \tt Y^\vv(k\dd+\vv u-\tau))\d u\\
&\quad+\int_0^{\ff{r-k\dd}{\vv}}\si_2(X^\vv_{k \dd}, \tt
Y^\vv(k\dd+\vv u-\tau))\d \tt W_2(u),
\end{split}
\end{equation}
where $\tt W_2(u):=(W_2(\vv u+k\dd)-W(k\dd))/\ss{\vv}$, which is
also a Wiener process. For fixed $\vv>0$ and $u\ge0$, let
\begin{equation*}
\bar Y^{X^\vv_{k\dd}}(u+\theta)=\tt Y^\vv(k\dd+\vv u+\theta), \ \ \
\theta\in[-\tau,0].
\end{equation*}
Then \eqref{q1} can be rewritten as
\begin{equation*}
\begin{split}
\bar Y^{X^\vv_{k\dd}}\Big(\ff{r-k\dd}{\vv}\Big) &=\tt
Y^\vv(k\dd)+\int_0^{\ff{r-k\dd}{\vv}}b_2\left(X^\vv_{k\dd},  \bar Y^{X^\vv_{k\dd}}(u), \bar Y^{X^\vv_{k\dd}}( u-\tau)\right)\d u\\
&\quad+\int_0^{\ff{r-k\dd}{\vv}}\si_2\left(X^\vv_{k\dd},\bar
Y^{X^\vv_{k\dd}}(u), \bar Y^{X^\vv_{k\dd}}( u-\tau)\right)\d \tt
W_2(u).
\end{split}
\end{equation*}
Consequently, by the weak uniqueness of solution, we arrive at
\begin{equation}\label{r2}
\mathscr{L}(\tt
Y^\vv_r)=\mathscr{L}\Big(Y^{X^\vv_{k\dd}}_{(r-k\dd)/\vv}(\tt
Y^\vv_{k\dd})\Big),
\end{equation}
where $\mathscr{L}(\zeta)$ denotes the law of random variable
$\zeta$.  Finally, we obtain from \eqref{d6}, \eqref{r1},
\eqref{r2},  and Lemma \ref{l5} that
\begin{equation*}
\begin{split}
\Upsilon_2(k,\dd,\vv)&\1
 (1+\E\|X_{k\dd}^\vv\|_\8^2 +\E\|\tt Y^\vv_{k\dd}\|_\8^2)\int_{k\dd}^{(k+1)\dd}\int_s^{(k+1)\dd}\exp\Big(-\ff{c(r-k\dd)}{\vv}\Big)\d
r\d s\\
&\1\vv\dd.
\end{split}
\end{equation*}
The 
proof is therefore complete.

\end{proof}

\begin{rem}
{\rm In this paper, we only focus on the case, where the diffusion
coefficient of the slow component is independent of the fast motion.
For the case that the slow component fully depends on the fast one,
there is an illustrative counterexample \cite[p.1011]{L10} in which
the weak convergence holds but there is no strong convergence.
}
\end{rem}

\begin{rem}
{\rm In the present paper, we explore a strong limit theorem for the
averaging principle for a class of two-time-scale SDEs with memory
under certain dissipative conditions. Nevertheless, our main result
can be generalized to some cases, where the fast motion does not
satisfy a dissipative condition. Indeed, by a close inspection of
the argument of Theorem \ref{main}, to cope with the non-dissipative
case, one of the crucial procedures is to discuss the ergodic
property of the frozen equation without dissipativity. However, for
some special cases, this problem has been addressed in Bao et al.
\cite{BYY13}.

}
\end{rem}

\begin{rem}
{\rm As we mentioned in the Introduction section, the study on
two-time-scale stochastic systems with memory  is still in its
infancy. So, there is numerous work to be done in the future. Here,
we list some of them. For the fast component, in this work we
concentrate on the case of point delay. So far, it seems hard to
extend our main result to the general case, e.g., the distributed
delay, where the main difficulty is to provide an error bound of the
difference in the strong sense between the fast component
$(Y^\vv(t))$ and its approximation $(\tt Y^\vv(t))$. Moreover, it is
also very challengeable to reveal the rate of strong convergence
established in Theorem \ref{main}  since the phase space of the
segment processes is infinite-dimensional. The questions above will
be addressed in our forthcoming work.

}
\end{rem}

\end{document}